\documentclass{article}
\usepackage{graphicx} % Required for inserting images

\usepackage[english]{babel}
\usepackage[utf8]{inputenc}
\usepackage[T1]{fontenc}

\usepackage[a4paper,top=3cm,bottom=2cm,left=3cm,right=3cm,marginparwidth=1.75cm]{geometry}

    \usepackage{amsthm}
    \usepackage{amssymb}
    \usepackage{amsmath}
    \usepackage{dirtytalk}
    \usepackage{mathtools}
    \usepackage{epsfig}
    \usepackage{caption}
    \usepackage[colorinlistoftodos]{todonotes}
    \usepackage[colorlinks=true, allcolors=blue]{hyperref}
    \usepackage{xcolor}
    \usepackage{pinlabel}
    \usepackage{geometry}
    \usepackage{subcaption}
    \usepackage{hyperref}
    \usepackage{authblk}
    \usepackage{longtable}
    \usepackage[shortlabels]{enumitem}

\theoremstyle{definition}
\newtheorem{definition}{Definition}[section]

\theoremstyle{plain}
\newtheorem{theorem}{Theorem}[section]
\newtheorem{lemma}{Lemma}[section]

\theoremstyle{remark}

\title{Knots and non-orientable surfaces in 3-manifolds}

\author{Alessia Cattabriga\footnote{Department of Mathematics, University of Bologna, piazza di Porta San Donato 5, 40126 Bologna - Italy alessia.cattabriga@unibo.it }, Paolo Cavicchioli\footnote{Inštitut za matematiko, fiziko in mehaniko (IMFM), Jadranska ulica 19, 1000 Ljubljana - Slovenja, paolo.cavicchioli@imfm.si
}, Rama Mishra\footnote{Department of Mathematics, Indian Institute of Science Education and Research, Pune, Maharashtra, 411008, India, r.mishra@iiserpune.ac.in}, Visakh Narayanan\footnote{Department of Mathematics, Indian Institute of Science Education and Research, Mohali, Punjab, 140308-India visakh@iisermohali.ac.in } }

\begin{document}
\maketitle
\begin{abstract}
 In this article, we propose a new approach for describing and understanding knots and links in a 3-manifold through the use of an embedded non-orientable surface. Specifically, we define a plat-like representation based on this non-orientable surface. The method applies to manifolds of the form  $M=\mathcal H\cup_{\varphi} \mathcal C(U)$ where $\mathcal H$ is a handlebody, $\mathcal C(U)$ is the mapping cylinder of the orientating two sheeted covering of a non-orientable closed surface $U$  and $\varphi:\partial \mathcal H\to \partial \mathcal C(U)$ is an attaching homeomorphism. We show that, by fixing such a splitting any link in the manifold can be represented as a plat-like closure of an element of the surface braid group of $\partial \mathcal H$. Manifolds of this type were extensively studied by J.H. Rubinstein \cite{rubinstein1978one}, where it is shown that any 3-manifold $M$, with a non-vanishing $H_2(M,\frac{\mathbb{Z}}{2\mathbb{Z}})$ will admit such a splitting. Thus the method is quite general. We provide explicit examples of such embeddings in lens spaces $L(2k,q)$ and the trivial circle bundles over orientable closed surfaces, $\Sigma\times S^1$. \\

\end{abstract}

\section{Introduction}

\par The triumph of classical knot theory mainly has been facilitated by the availability of several methods of representing knots in $S^3$. Among these representations, braid groups have played a pivotal role. The use of braid groups to represent knots and links
dates back to the work of Alexander \cite{alexander1923lemma} and Markov \cite{markov1935freie} for the standard closure and Hilden \cite{hilden1975generators} and Birman \cite{birman1976stable} for the plat closure. Both representations have been successfully generalized to closed, connected, orientable 3-manifolds \cite{bellingeri2012hilden, cattabriga2018markov,  cattabriga2008extending,  cavicchioli2021algorithmic, cavicchioli2023mixed,diamantis2015braid,doll1993generalization, diamantis2015braid,lambropoulou1997markov,lambropoulou2006algebraic,skora1992closed}. These can be  highly effective in the study of knots and links, as shown by various authors, see for example, \cite{cattabriga2006alexander,diamantis2016topological,gabrovsec2011homflypt}). In particular, the generalization of plat closures relies on Heegaard splittings, which decompose a 3-manifold into two handlebodies by means of an embedded, closed, connected orientable surface.\\
 
 More recently, a new representation involving non-orientable surfaces was introduced in \cite{mishra2023plat}. There, the authors prove that links in \(\mathbb{R}P^3\) can be \say{sewn} around the projective plane at infinity into a plat-like structure. Building on this foundation, the goal of this article is provide a plat representation for knots in a large class of closed, connected, orientable 3-manifolds that contain non-orientable surfaces. The non-orientable surfaces embedded in lens spaces and in manifolds of the form $\Sigma\times S^1$, where $\Sigma$ is a closed connected orientable surface were studied by Bredon and Wood \cite{bredon1969non} using the work of Rene Thom. The same results were reproduced using simpler methods by W. End \cite{end1992non}. It was observed by Rubinstein \cite{rubinstein1978one,rubinstein1979,rubinstein1982} that the embeddings of non-orientable surfaces in a 3-manifold can be very useful in extracting topological information. Rubinstein defined a notion of a \say{one-sided Heegaard splitting} where a 3-manifold can be constructed from gluing the boundaries of a handlebody $\mathcal{H}$ and the mapping cylinder $\mathcal{C}(U)$ of the orientating two-sheeted cover of a closed non-orientable surface, $U$ and showed that any 3-manifold $M$ with a non-vanishing $H_2(M, \frac{\mathbb{Z}}{2\mathbb{Z}})$ will admit such a splitting. We refer to a 3-manifold which has such a splitting as \say{splittable}.\\
 
 \par In this article, we exploit such a splitting to develop a plat-like description for knots and links in $M$. We exhibit  that the embeddings of non-orientable surfaces in lens spaces of type $L(2k,q)$ and the manifolds of the type $\Sigma\times S^1$ described in \cite{bredon1969non, end1992non} induce one-sided Heegaard splittings. Since there are 3-manifolds which does not contain any non-orientable surface, such as $S^3$ or more generally lens space of the type $L(2k+1,q)$ it is clear that not all 3-manifolds are splittable. The focus of this article is to describe the topological aspects of this plat representation. A detailed algebraic study of the plat-like representation using surface braid groups associated with the splitting surface $\partial \mathcal H\subset M$ will be presented in a subsequent paper as well as a more systematic study of splittable 3-manifolds. \\

\noindent \textbf{Organization of the paper:} In Section \ref{section_ourmanifolds}, we will recall some definitions regarding splittable 3-manifolds. In Section \ref{section_nonorientable}, we define our plat-like representation for links in a  splittable 3-manifold and we establish any link in this manifold is isotopic to a plat closure of this form. In Sections \ref{section_lens_spaces} and \ref{surfacebudlesection} respectively, the examples of lens spaces  and 3-manifolds of the type $\Sigma \times S^1$, with $\Sigma$ a closed connected orientable surface, are discussed. 

\section{One-sided Heegaard splittings}
\label{section_ourmanifolds}
\par By Heegaard's theorem, all closed connected orientable 3-manifolds can be constructed by gluing two handlebodies along their boundaries. Rubinstein \cite{rubinstein1978one} generalized this to a splitting of the type $M=\mathcal{H}\cup_\varphi \mathcal{C}(U)$, where a 3-manifold is constructed by  gluing a handlebody $\mathcal H$ with the mapping cylinder $\mathcal C=\mathcal C(U)$ of the orientating double cover of a closed connected non-orientable surface $U$. Rubinstein called this decomposition as a \say{one-sided Heegaard splitting}. In this article a manifold admitting such a splitting is called \say{splittable}. Formally, we may define these as follows.\\

\begin{definition}
    A closed orientable 3-manifold $M$ is called splittable if there exists an embedded non-orientable surface $U$ in $M$ such that $M\setminus\textup{int}(\mathcal N(U))$ is homeomorphic to a handlebody, with $\mathcal N(U)$ a (closed) tubular neighborhood of $U$ in $M$.
\end{definition}

Since a tubular neighbourhood of any non-orientable surface $U$ in an orientable 3-manifold $M$ is homeomorphic to the mapping cylinder of the canonical orientating double cover $\pi:\Sigma\to U$, with $\Sigma$ a closed connected orientable surface, it follows that $M$ admits a splitting $\mathcal{H}\cup_\varphi \mathcal{C}(U)$ with $\varphi: \partial \mathcal H\to\Sigma$ a gluing homemomorphism  classified by the isotopy classes of self homeomorphisms of $\Sigma$. Note that if $U$ has (non-orientable) genus $g$, say $U=U_g$, then the two sheeted covering $\Sigma$ has (orientable) genus $g-1$, say $\Sigma=\Sigma_{g-1}$; so, since we want $\mathcal C$ and  $\mathcal H$ to have homemorphic  boundaries we will glue $\mathcal C(U_{g})$ with $\mathcal H_{g-1}$ using  (an isotopy class of) a  self homeomorphisms of $\Sigma_{g-1}$, a closed connected orientable surface of genus $g-1$.\\

One might what kinds of 3-manifolds can be obtained using by this method. Clearly, not all closed, connected, orientable 3-manifolds fall into this category, as they must contain an embedded non-orientable surface.  For example, lens spaces  $L(p,q)$ with $p$ odd (see \cite{bredon1969non}) cannot be obtained using this approach. This raises the question: do all closed, connected, orientable 3-manifolds containing a closed, connected, non-orientable submanifold admit the decomposition described above? In \cite[Theorem 1]{rubinstein1978one} Rubinstein gives the following  answer.
\begin{theorem}[\cite{rubinstein1978one}]\label{teo_rubinstein}
Let $M$ be a three manifold. For every non-trivial cycle $\alpha$ in $H_2(M,\frac{\mathbb{Z}}{2\mathbb{Z}})$ there is an embedding of a non-orientable surface $U$ such that $[U]=\alpha$ with $M\setminus U$ homeomorphic to an open handlebody.
\end{theorem}

Thus, 3-manifolds with non-vanishing $H_2(M,\frac{\mathbb{Z}}{2\mathbb{Z}})$, are splittable. Moreover, in \cite[Corollary 4.5]{end1992non} the following result is proved.

\begin{theorem}[\cite{end1992non}] If $i:U_g\hookrightarrow M$ is an embedding of a closed connected non-orientable surface $U_g$ in a 3-manifold $M$ with a non-trivial normal bundle then $i_*: H_2(U_g,\frac{\mathbb{Z}}{2\mathbb{Z}})\rightarrow H_2(M,\frac{\mathbb{Z}}{2\mathbb{Z}})$ is non-zero.
\end{theorem}

Since the 3-manifolds we are interested in are orientable, the normal bundle of $i(U_g)$ is non-trivial. As consequence all closed, connected, orientable 3-manifolds containing an embedded  closed  connected non-orientable surface are splittable.  \\

Important examples of splittable 3-manifolds are  lens spaces $L(2k,q)$ and trivial circle bundles over closed orientable surfaces, $\Sigma_g\times S^1$. In Section \ref{section_lens_spaces} and Section \ref{surfacebudlesection} we describe the embeddings of non-orientable surfaces constructed by Bredon and Wood \cite{bredon1969non,end1992non} inducing  one-sided splittings.  \\

\par However, the family of splittable manifolds is significantly larger, as we have considerable flexibility in choosing the gluing homeomorphism $\varphi: \partial \mathcal{H}\to \partial \mathcal C (U)$. To illustrate this, we will examine an example where $U$  is a Klein bottle. This is the first interesting case, in terms of genus, since when $g=1$ the surface $U$ is the projective plane and $\partial \mathcal C(U)$ is the 2-sphere. In this case, due to the Alexander's trick, there is only one way (up to isotopy) to glue the two pieces, resulting in the projective space $\mathbb{RP}^3$. 

\begin{figure}
    \centering
    \includegraphics[width = .9\textwidth]{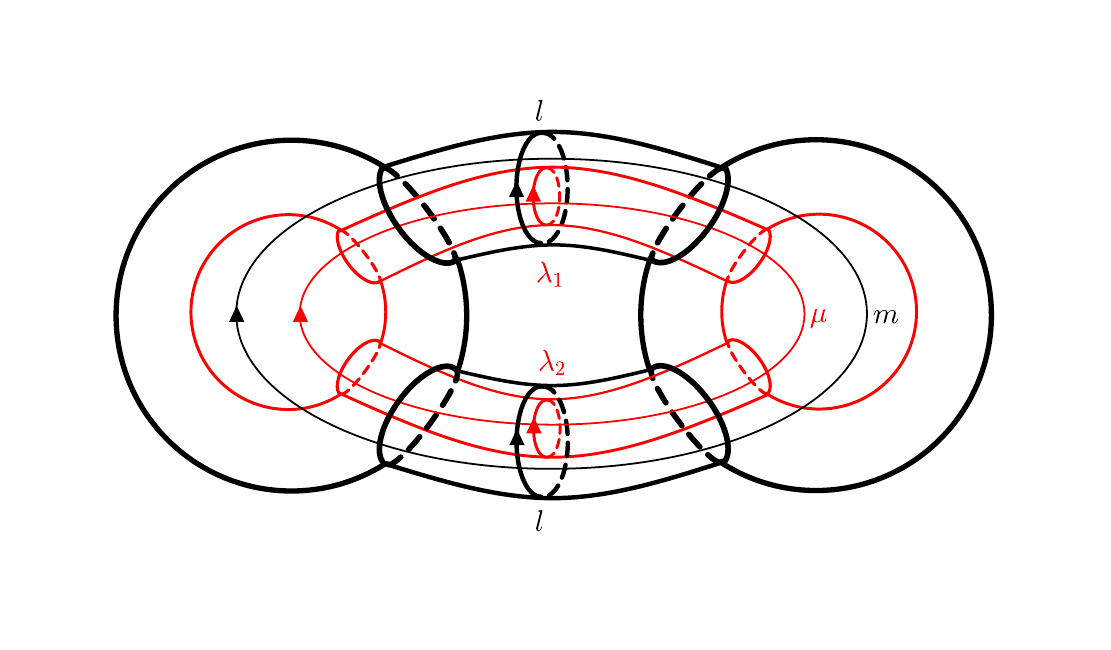}
    \caption{Mapping cylinder of the Klein bottle.}
    \label{fig:mobius_example}
\end{figure}

If $g=2$, the mapping cylinder $\mathcal C=\mathcal C(U)$ is the two sheeted covering map from the torus to the Klein bottle, so it is a twisted $I$-bundle over the Klein bottle. In Figure \ref{fig:mobius_example} there is a schematic description of $\mathcal C$ where  the external black torus represents $\partial C$; moreover, one may think of the Klein bottle as the result of  the quotient of the internal red torus by a symmetry that identifies the right sphere with the left one followed by an antipodal map on the circular fibres of the two central cylinder.  In such a way each cylinder becomes a  Möbius bands glued along the boundary circle to a two-holed sphere. If $\lambda_1$ and $\lambda_2$ denote the cores of the two Möbius bands,   the fibres in the mapping cylinder of all the points  on $\lambda_i$  will constitute another Möbius band, with its boundary lying on the boundary torus of $\mathcal C$. In other words, $\lambda_1$ and $\lambda_2$ are covered by homologous circles on the torus, which are essential. We may call  this curve the \say{longitude} of the mapping cylinder and denote the homology class of it with $l$ (oriented as in figure). There is a circle $\mu$,  unique up to isotopy, on the Klein bottle, which intersects  the M\"obius bands at exactly one point and does not bound a disk. In the mapping cylinder, the fibres of all points on $\mu$ form a M\"obius band, whose boundary lies on the boundary torus. This circle on the torus is homologically essential and intersects the longitude $l$, exactly once. We will call this curve the \say{meridian} of the mapping cylinder and denote its homology class with $m$  (oriented as  in figure). Moreover, we still denote with $\lambda_1$, $\lambda_2$ and $\mu$ the homology classes of the corresponding oriented curves. \\

\par Note that  $l$ and $m$ generate the first homology group of the boundary torus, while the homology classes of $\lambda_1$, $\lambda_2$ and $\mu$ generate the first homology group of the mapping cylinder. Let $\mathcal V$ be a solid torus and let  $l'$ and $m'$ be a standard basis for the first homology group of $\partial \mathcal V$.  When we glue the boundaries of the solid torus $\mathcal V$ and the mapping cylinder $\mathcal C$, through $\varphi: \partial \mathcal V\to \partial \mathcal C$ to obtain $M$, the images of $l'$ and $m'$ will be some integral linear combinations of  $l$ and $m$. This can be expressed as a $2\times 2$ matrix
\[
\begin{bmatrix}
 	l'&\\
 	m'
\end{bmatrix} 
=
\begin{bmatrix}
 	p_l& q_l\\
 	p_m& q_m
\end{bmatrix}
\begin{bmatrix}
 	l&\\
 	m
\end{bmatrix}
\]
where $p_l q_m - p_m q_l =-1$, describing the map induced by $\varphi $ on the first homology group and encoding all the relevant information to determine $M=\mathcal V\cup_{\varphi}\mathcal C$. Then, a part of the Mayer-Vietoris sequence for the glued manifold $M$ looks like\\

\begin{equation*}
H_1(\mathcal V\cap \mathcal C)\stackrel{f}{\longrightarrow} H_1(\mathcal V)\oplus H_1(\mathcal C) \longrightarrow H_1(M)\longrightarrow H_0(\mathcal V \cap \mathcal C )
\end{equation*}
\begin{equation*}
\langle l,m \rangle \longrightarrow \langle l'\rangle \oplus \langle \lambda_1, \lambda_2, \mu \rangle /\langle  2(\lambda_1-\lambda_2)\rangle    \longrightarrow H_1(M)\longrightarrow \mathbb{Z},
\end{equation*}
where $\langle t_1,\ldots,t_k\rangle$ denotes the free Abelian group generated by $t_1,\ldots,t_k$. The last map, sends every 1-cycle in $M$  to $0$,  thus the previous map should be surjective. Hence we have
\begin{equation*}
H_1(M)\cong H_1(\mathcal V)\oplus H_1(\mathcal C)/ G
\end{equation*}
where $G$ is the image of $H_1(\mathcal V\cap\mathcal  C)$ under the map $f$. Now if we choose $\varphi$ so that the associated matrix in homology is, 
$$ \begin{bmatrix}
 	0& 1\\
 	1& 0
\end{bmatrix}
$$
we get,
\begin{align*}
f(l)&= (l', 2\mu),\\
f(m)&= (0, 2\lambda_1).\\
\end{align*}
So we have
\begin{align*}
G&= \langle l' \rangle \oplus \langle 2\lambda_1 \rangle \oplus \langle 2\mu  \rangle ,
\end{align*} 
which implies
\begin{equation*}
H_1(M)\cong \mathbb{Z}\oplus \frac{\mathbb{Z}}{2\mathbb{Z}}\oplus \frac{\mathbb{Z}}{2\mathbb{Z}}.
\end{equation*}
So, $M$ is not a lens space or a trivial surface bundle $\Sigma\times S^1$, since the first homology group of neither of these can have a subgroup isomorphic to the Klein group $\frac{\mathbb{Z}}{2\mathbb{Z}}\oplus \frac{\mathbb{Z}}{2\mathbb{Z}}$.  \\

Since the mapping class group of the torus is quite large, we can construct several interesting subgroups of $H_1(\mathcal V)\oplus H_1(\mathcal C)$ as the image of the gluing map. As a result, the family of manifolds that can be constructed is remarkably rich. The classification of all such manifolds, however, lies beyond the scope of this paper, and we do not intend to address it here.

\section{Non-orientable plat closure}
\label{section_nonorientable}

In the previous section, we described the class of manifolds we aim to consider. In this section, we introduce the notion of a non-orientable plat closure for knots and links in such manifolds.\\

\begin{figure}
    \centering
    \includegraphics[width=0.5\linewidth]{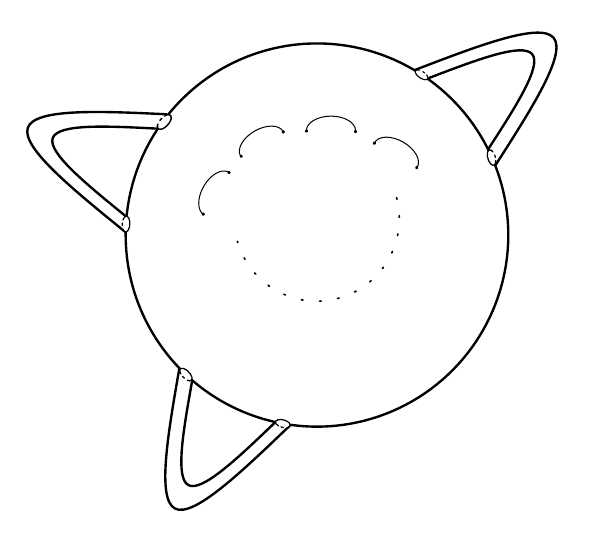}
    \caption{Capping curves inside a handlebody. }
    \label{fig:cap_handlebody}
\end{figure}

Let $M$ be a splittable manifold with (non-orientable) splitting surface $U$.  To define the concept of  non-orientable plat closure for a link in $M$, we first need to introduce the idea of capping curves in  $\mathcal H$ and in $\mathcal N(U)$. These are special \say{trivial}  arcs into which a link in non-orientable plat position will be cut by the orientable surface $\partial \mathcal H=\partial \mathcal N(U)$.   We start defining the capping curves in the handlebody. Recall that an arc $\gamma$ in a manifold $\mathcal K$  with non-empty boundary is \textit{neatly embedded}  if  $\partial \gamma\subset \partial \mathcal K$ and it is  \textit{boundary parallel} if it is neatly embedded and there exists a disk $D$, called \textit{trivializing disk}, embedded in $\mathcal K$ such that $\gamma=\gamma \cap D=\gamma\cap \partial D$ and  $\partial D\setminus \gamma\subset \partial \mathcal K$.  Capping curves in the handlebody are unlinked boundary parallel arcs, that is boundary parallel arcs with  disjoint trivializing disks (see Figure \ref{fig:cap_handlebody}). 

% \begin{figure}
%     \centering
%     \includegraphics[width=0.5\linewidth]{CAPPING_HANDLEBODY.PNG}
%     \caption{SOMETHING TO HELP THE READER - MOBIUS BAND WITH THE LONGITUDES}
%     \label{fig:cap_handlebody}
% \end{figure}

Let's pass to  the tubular neighborhood $\mathcal N(U)$. As capping curves we will consider curves that generalize the notion of \textit{residual} curves introduced in \cite{mishra2023plat} for the case of the projective space $\mathbb{RP}^3$.   Before defining them precisely we need a preparatory lemma.

\begin{lemma}
Let $U_g$ be a genus $g$ closed connected non-orientable surface embedded in a closed connected orientable manifold $M$.  We have,
\begin{equation*}
H_1(\mathcal N(U_g),\partial \mathcal N(U_g))\cong \underbrace{\frac{\mathbb{Z}}{2\mathbb{Z}} \oplus \frac{\mathbb{Z}}{2\mathbb{Z}}\oplus ...\oplus \frac{\mathbb{Z}}{2\mathbb{Z}}}_{g\ \textup{times}}.
\end{equation*}

\end{lemma}
\begin{proof}
In analogy with the description of the mapping cylinder of the orientating two-sheeted covering of the Klein bottle discussed in the previous section, we want to describe a model for $\mathcal N(U)$.  To start with, note that the connected sum of a surface $\Sigma$ with a projective plane, is homeomorphic to the surface obtained by identifying the boundary circle of $\Sigma \setminus \textup{int}(D^2)$, with $D^2$  an embedded closed disc,  with the boundary circle of a M\"obius strip.\\

If $g=1$ then  $U_1\cong\mathbb{RP}^2$ and we want to describe it  as the quotient of the 2-sphere by a $\frac{\mathbb{Z}}{2\mathbb{Z}}$ action. Consider the surface, homeomorphic to the 2-sphere, obtained by  connecting two 2-spheres with a tube, as in Figure \ref{cover1}. The shaded region  is a cylinder $C\cong S^1\times [0,1]$ whose complement is the disjoint  union of two copies of $S^2\setminus D^2$. We identify these two components  by a reflection on a plane containing the core circle $S^1\times \{\frac1 2\}$, while the action of $\frac{\mathbb{Z}}{2\mathbb{Z}}$ on $C$ is the composition of the reflection along the plane with the antipodal map in each of the circles $S^1\times \{*\}$. In this way  $C$  will become  a M\"obius strip, and the whole surface, shown in Figure \ref{cover1}, will become the surface obtained by gluing the boundary of a M\"obius strip to the boundary of $S^2\setminus D^2$, which is exactly the projective plane $U_1\cong\mathbb{R}P^2$. \\  
Now by thickening the initial surface (as shown in Figure \ref{bundle1}) and taking the quotient  on one of its boundary components, we obtain $\mathcal N(U_1)$. 
\begin{figure}
\centering
\includegraphics[scale=0.6]{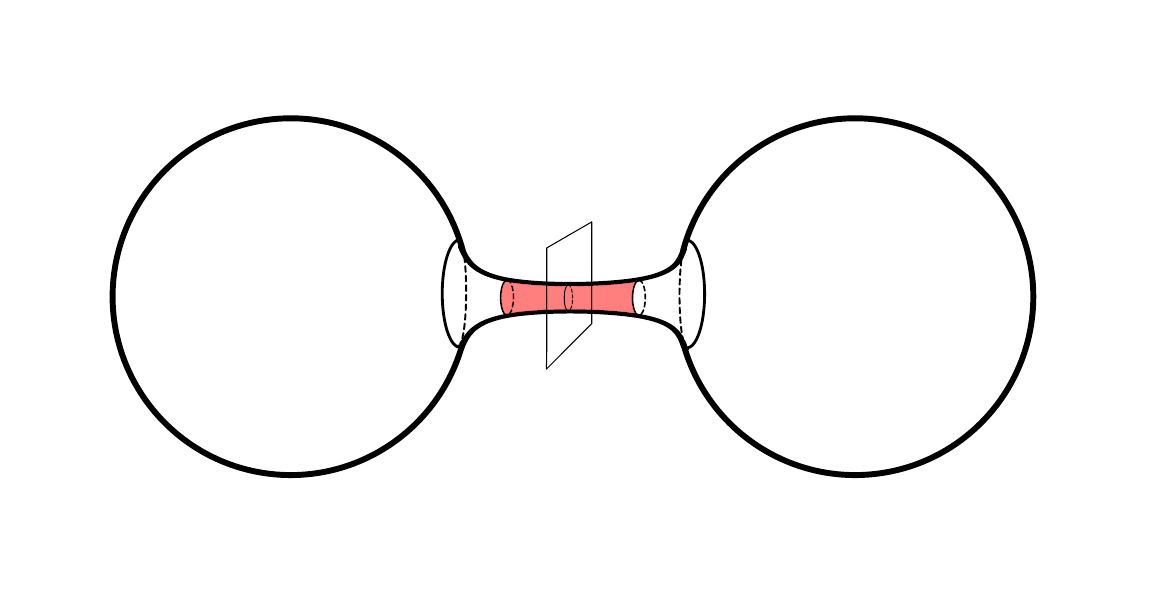}
\caption{Two spheres with connected with a tube. }
\label{cover1}
\end{figure}

\begin{figure}
\centering
\includegraphics[scale=0.6]{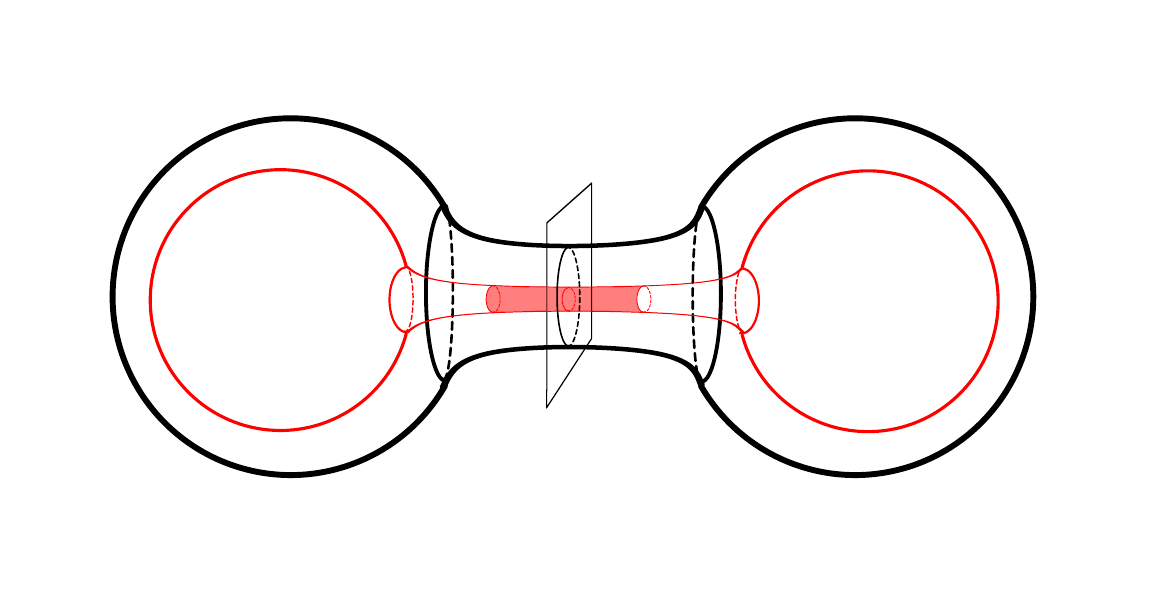}
\caption{Thickening two spheres connected by a tube.}
\label{bundle1}
\end{figure}

\begin{figure}
\centering
\includegraphics[scale=0.35]{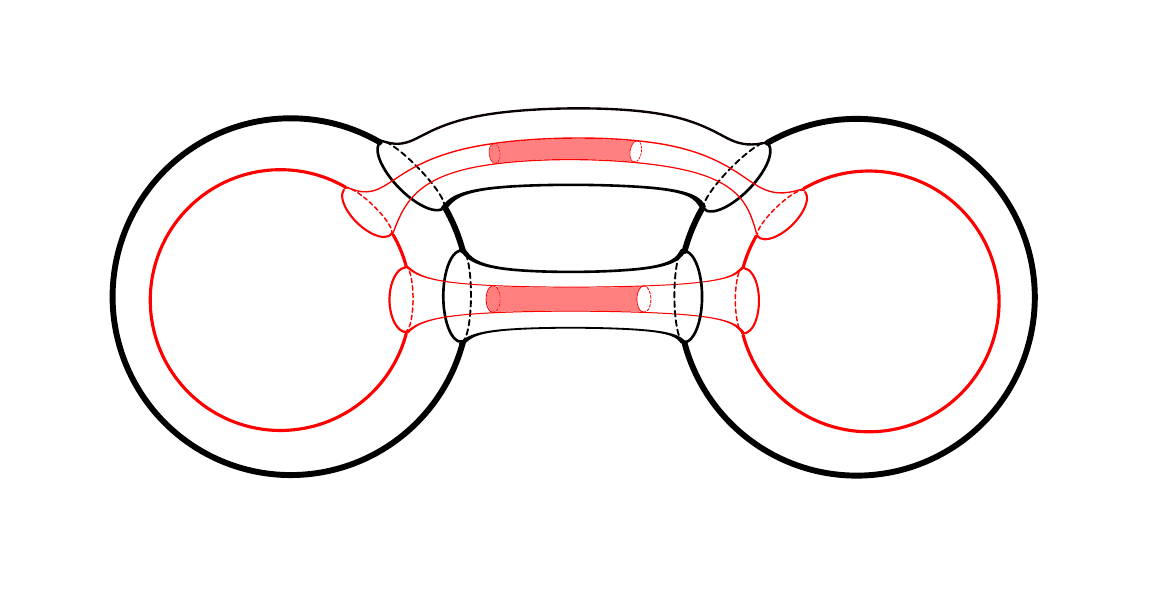}
\includegraphics[scale=0.35]{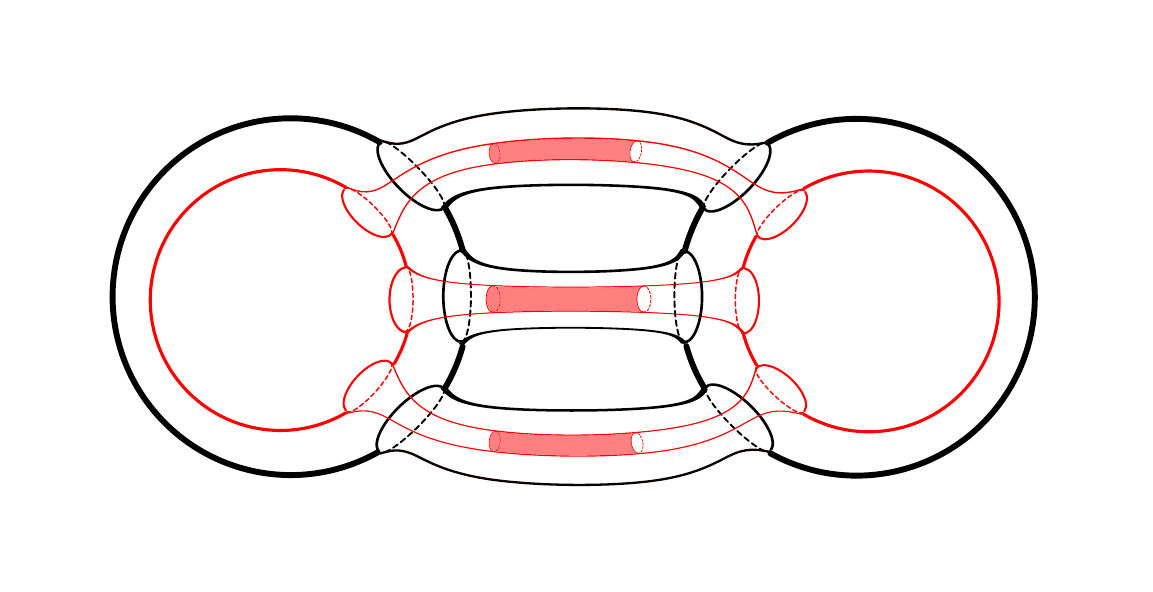}
\caption{Thickenings of the cases $g=2$ and $g=3$.}
\label{higherbundles}
\end{figure}

If in the initial surface we consider  $g$ tubes, instead of just one,  connecting the two spheres, and we repeat the previous procedure for each tube we obtain a description of $\mathcal N(U_g)$.  Figure \ref{higherbundles} shows a schematic for the cases $g=2$ and $3$. \\

Notice that, since $\mathcal N(U_g)$ retracts to $U_g$, then $H_1(\mathcal N(U_g)) \cong H_1(U_g)$. By construction, the core circles of the shaded cylinders are two sheeted covers of the cores of the corresponding M\"obius strips and each of them  generates an infinite cyclic group in  $H_1(\mathcal N(U_g))$, while the sum of all the core circles will generate a $\frac{\mathbb{Z}}{2\mathbb{Z}}$ component. So, we may choose the cores of $g-1$ M\"obius strips and the sum of all of the cores as the generating cycles of
\begin{equation*}
H_1(\mathcal N(U_g))\cong H_1(U_g)\cong\mathbb{Z}^{g-1}\oplus \frac{\mathbb{Z}}{2\mathbb{Z}}.
\end{equation*}
In the relative homology group $H_1(\mathcal N(U_g), \partial   \mathcal N(U_g))$,  the cores of all of the cylinders are trivial cycles, since they are homologous to a circle in $\partial \mathcal N(U_g)$ as shown in Figure \ref{cyclecancel}. 
Thus every generator of $H_1(\mathcal N(U_g))$ becomes a generator of a $\frac{\mathbb{Z}}{2\mathbb{Z}}$,  and hence
\begin{equation*}
    H_1(\mathcal N(U_g),\partial \mathcal N(U_g))\cong \frac{\mathbb{Z}}{2\mathbb{Z}}\oplus \frac{\mathbb{Z}}{2\mathbb{Z}}...\oplus \frac{\mathbb{Z}}{2\mathbb{Z}}.
\end{equation*}
\begin{figure}
\centering
\includegraphics[scale=0.6]{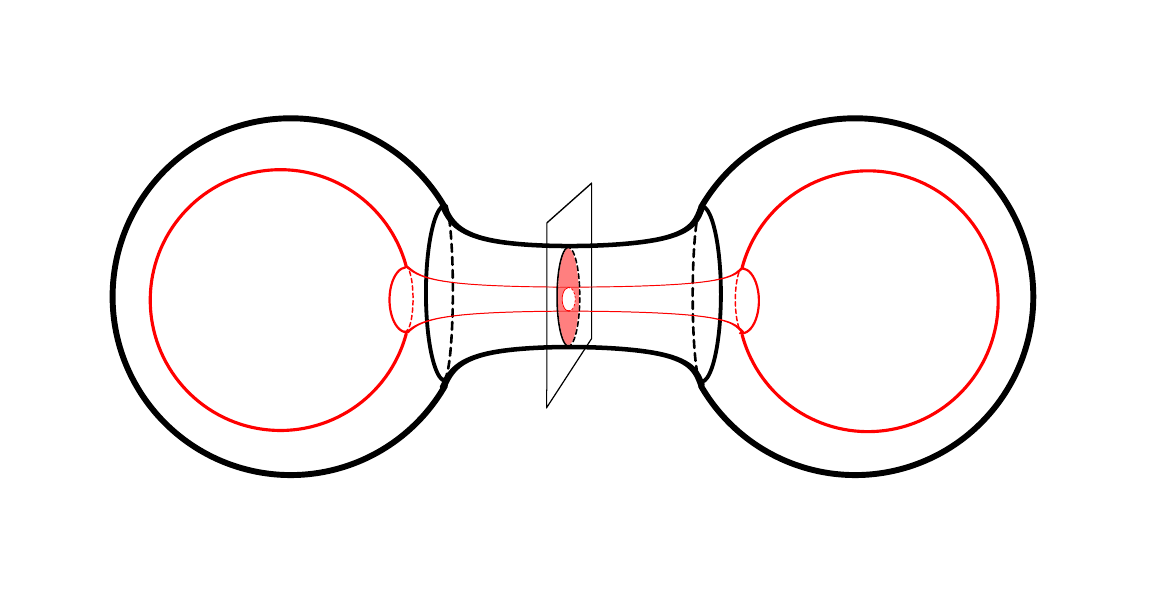}
\caption{Cancellation of the cores of the cylinders.}
\label{cyclecancel}
\end{figure}

\end{proof}

Figure \ref{capcur} represents  one of the thickened tubes connecting two of the spheres depicted in Figure \ref{higherbundles}. As described in the previous proof, the shaded cylinder will be identified into a M\"obius strip in order to construct $N(U_g)\cong\mathcal C$. Then the points  $Q$ and $Q'$ will be identified and  the line segment from $P$ to $Q$ and the from $Q'$ to $R$ will be joined to one curve connecting the points $P$ to $R$. The points $P$ and $R$ lie on the boundary of $\mathcal N(U_g)$, and thus this curve represents a cycle in $H_1(\mathcal N(U_g),\partial \mathcal N(U_g))$. It is obvious that this cycle is not homologically trivial and  it is a generator of this homology group. 
We will call each such curve as a \say{capping curve} in $\mathcal C(U_g)\cong \mathcal N(U_g)$. By a collection of capping curves, we always refer to a set of mutually unlinked such curves: that is, for each couple of capping curves  there exists a parametrization of $\mathcal N(U_g)$ as a twisted bundle over $U_g$ so that the two capping curves are fibres. \\

\begin{figure}
\centering
\includegraphics[width = .7\textwidth]{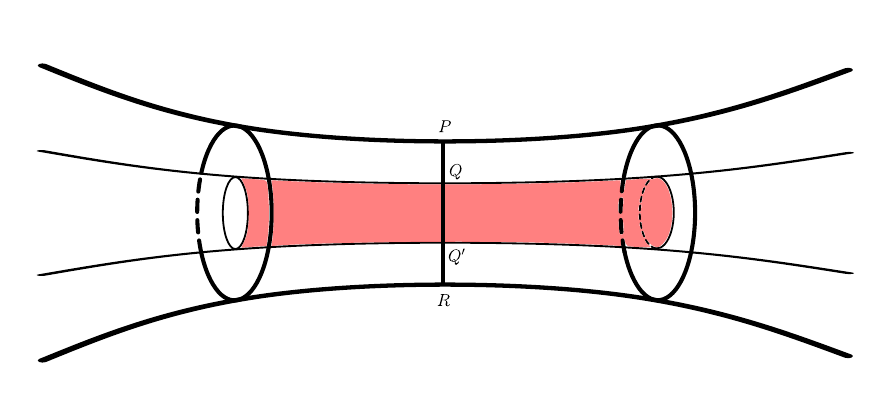}
\caption{A capping curve in $\mathcal C$.}
\label{capcur}
\end{figure}

Now we are ready to give the definition of non-orientable plat closure in $M$.

\begin{definition}
Let $M$ be a closed connected orientable 3-manifold containing a non-orientable embedded surface  $U_g$ of genus $g$ such that $M\setminus \mathcal N(U_g)$ is an handlebody $\mathcal{H}_{g-1}$. By  thickening the surface $\Sigma_{g-1}=\partial \mathcal{H}_{g-1}=\partial \mathcal N(U_g)$ we can decompose $M$ as 

$$M=\mathcal{H}_{g-1}\cup \Sigma_{g-1}\times I\cup \mathcal N(U_g).$$ 
Let $\beta\in B_{g-1,n}$ be a braid in  $\Sigma_{g-1}\times I$. Any link formed by joining the ends of $\beta$ to some collection of capping curves in $\mathcal H_{g-1}$ and $\mathcal N(U_g)$ will be called the \textit{non-orientable plat closure} of the braid $\beta$ and will be denoted as $\hat{\beta}$. 
\end{definition}

\par As we will show in Sections \ref{section_lens_spaces} and \ref{surfacebudlesection}, lens spaces of the type $L(2k,p)$ and trivial circle bundles $\Sigma_g\times S^1$ are splittable manifolds. This immediately raises a question.\\

\noindent\textit{Is every link in a splittable manifolds $M$, isotopic to the non-orientable plat closure of some braid?}\\

\par In  \cite{cattabriga2018markov} the authors answer a similar question for the case of orientable plat closure, that is, corresponding to the decomposition 
$$M=\mathcal{H}\cup \Sigma\times I \cup \mathcal{H}', $$ 
where $\mathcal{H}$ and $\mathcal{H}'$ are both handlebodies with the same genus and $\Sigma$ is Heegaard surface for $M$. In \cite{mishra2023plat} the authors study the case of 
$$\mathbb{RP}^3=B^3\cup S^2\times I\cup \mathcal N(\mathbb{RP}^2),$$
with $B^3$  a ball and $S^2$ its boundary sphere. 
We answer this question completely in the following theorem.

\begin{theorem}
Every link in a splittable 3-manifold $M$ is isotopic to the non-orientable plat closure of a braid in $\Sigma_{g-1}$. 
\end{theorem}

\begin{proof}
Consider the splitting $M=\mathcal H_{g-1}\cup \Sigma_{g-1}\times I\cup \mathcal N(U_g)$. Parametrize $I$ so that  $\Sigma_{g-1}\times {0}=\partial \mathcal H_{g-1}$ and $\Sigma_{g-1}\times {1} = \partial \mathcal N(U_g)$. Also let $h: \Sigma_{g-1}\times I \to I$ be the natural projection. Given a link, $L\subset M$, it may intersect the three regions in this splitting. \\

\begin{figure}
\centering
\includegraphics[width = \textwidth]{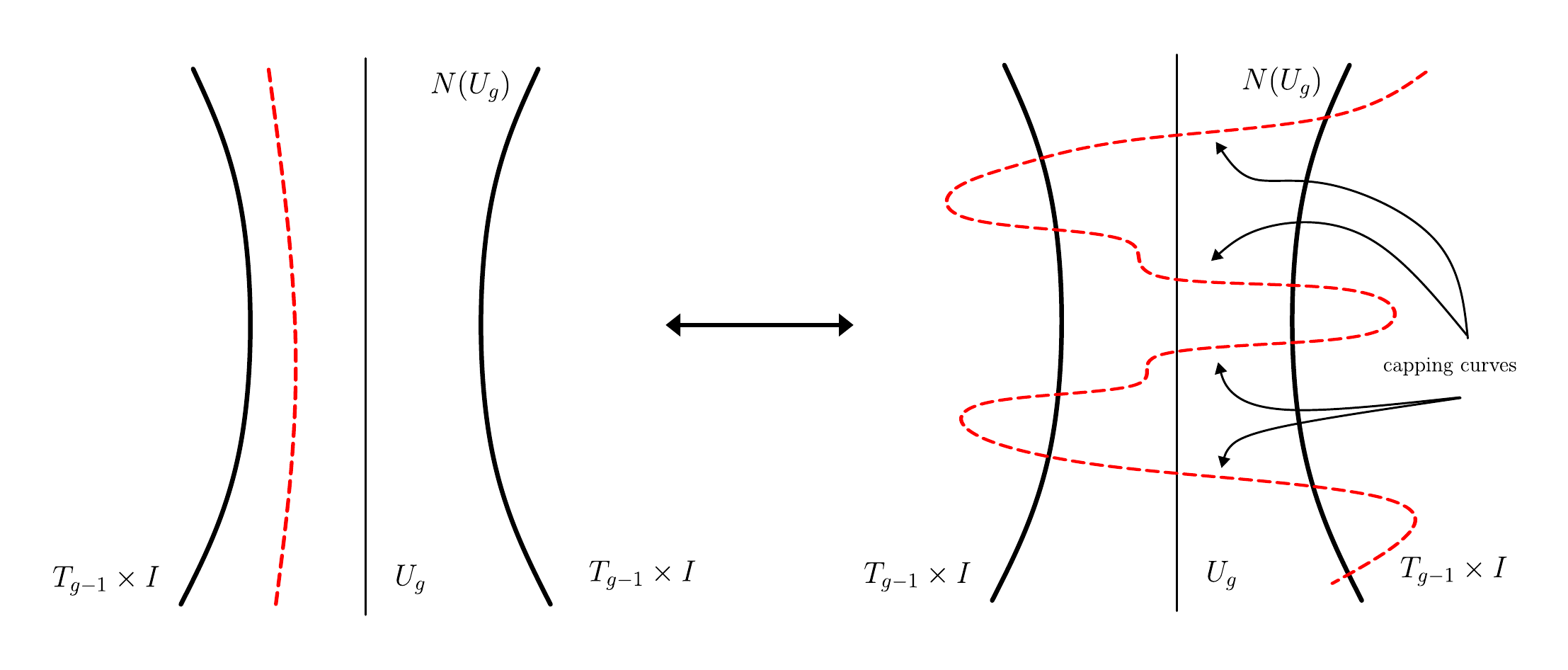}\\
\caption{Sliding an arbitrary curve to obtain capping curves in $\mathcal N(U_g)$ and boundary parallel curves in $\Sigma_{g-1}\times I$.}
\label{slidenug_total}
\end{figure}

\par As depicted in Figure \ref{slidenug_total} every boundary parallel arc $\gamma\subset \mathcal N(U_g)$ can be isotopically slided to a curve $\gamma'\subset M$, so that $\gamma' \cap \mathcal N(U_g)$ is a collection of capping curves and $\gamma'\cap \Sigma_{g-1}\times I$ is a collection of boundary parallel arcs. \\

\par The curves inside $\Sigma_{g-1}\times I$ can be classified as follows. If a curve contains no critical  points of $h$, we say that it is monotonic. If a monotonic curve has its ends on $\partial \Sigma_{g-1}\times I$ then, it will be called braided. A generic curve may contain many critical points. We may divide any such curve into several parts, by introducing some points in between so that, between any two adjacent points, the curve is monotonic or it has exactly one critical point. As depicted in Figure \ref{slidehg_total}, we may slide any minima point into $\Sigma_{g-1}\times I \cup \mathcal H_{g-1}$ so that the part inside $\mathcal H_{g-1}$ is a boundary parallel arc and the part inside $\Sigma_{g-1}\times I$ consists of two monotonic arcs.\\

\begin{figure}
\centering
\includegraphics[width = .5\textwidth]{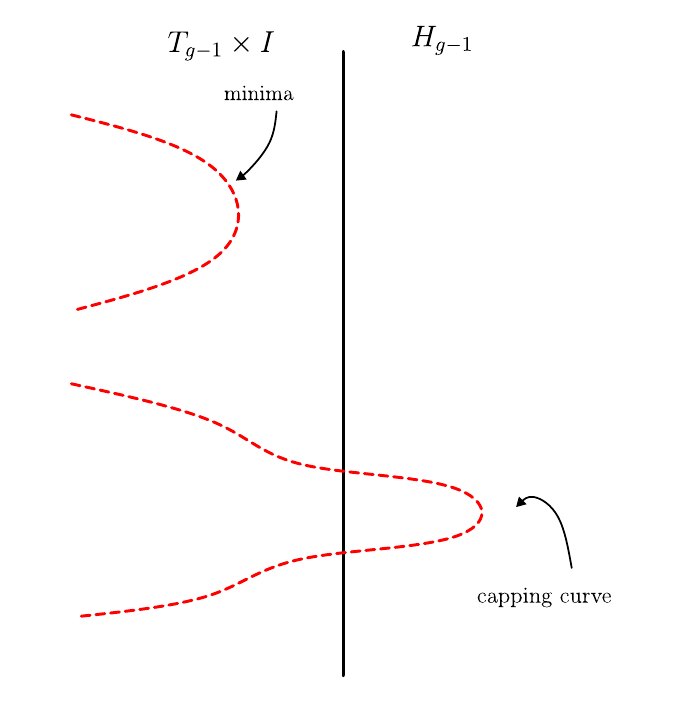}\\
\caption{Sliding minima points in $\Sigma_{g-1}\times I$ into $\mathcal H_{g-1}$.}
\label{slidehg_total}
\end{figure}

\par  We may slide any maxima point  to $\mathcal N(U_g)$ so that we get two monotonic arcs in $\Sigma_{g-1}$ and a boundary parallel arc in $\mathcal N(U_g)$. Now as described above, we may again slide that boundary parallel arc in $\mathcal N(U_g)$ so that we get a collection of capping curves in $\mathcal N(U_g)$ and some boundary parallel arcs in $\Sigma_{g-1}\times I$ (see Figure \ref{slidetg}). Notice that since each such boundary parallel arc in $\Sigma_{g-1}\times I$ is near $\partial \mathcal N(U_g)= \Sigma_{g-1}\times {1}$, we may assume that they contain exactly one minima point. And then we may slide each of these so that we get two braided curves in $\Sigma_{g-1}\times I$ and a boundary parallel curve in $\mathcal H_{g-1}$. If there is a point of inflection, then we may modify that part of the curve so that we get a maxima and minima. \\

It is clear that by sliding curves if needed, we may assume that all the boundary parallel arcs inside $\mathcal H_{g-1}$ are mutually unlinked. That is, we may assume that we always have a collection of capping curves in $\mathcal H_{g-1}$ after sliding the critical  points in $\Sigma_{g-1}\times I$.  \\

\begin{figure}
\centering
\includegraphics[width = .5\textwidth]{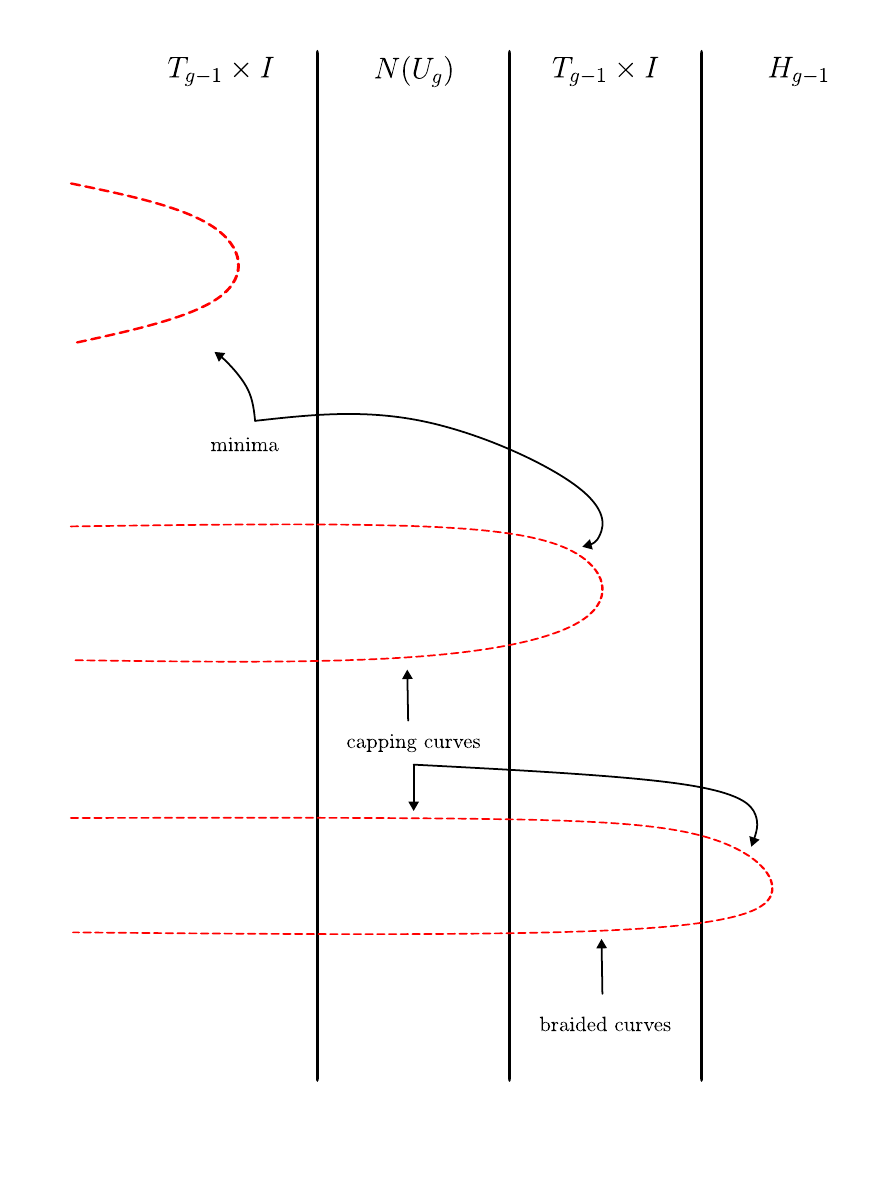}\\
\caption{A maxima point in $\Sigma_{g-1}\times I$. Sliding it to obtain capping curves in $\mathcal N(U_g)$ and a minima point in $\Sigma_{g-1}\times I$, then sliding the minima point from $\Sigma_{g-1}\times I$ to obtain braided curves in $\Sigma_{g-1}\times I$ and a capping curve in $\mathcal H_{g-1}$.}
\label{slidetg}
\end{figure}

\par Thus we can isotopically modify any curve in $\Sigma_{g-1}\times I$ into a collection of capping curves in both $\mathcal H_{g-1}$ and $\mathcal N(U_g)$ and braided curves in $\Sigma_{g-1}\times I$. Notice that now the part inside $\Sigma_{g-1}\times I$ is a braid and the whole link looks like its non-orientable plat closure.

\end{proof}

\section{Lens spaces} \label{section_lens_spaces}
%\subsection{The Bredon-Wood embeddings}
In this section we deal with lens spaces recalling results from  in \cite{bredon1969non,rubinstein1978one,end1992non} and working out explicitly some families of examples.\\

The embeddability of non-orientable surfaces in  lens spaces was originally studied in \cite{bredon1969non,end1992non}. The authors, using the fact that $H^1(L(p,q);\mathbb Z_2)\cong \textup{Hom}(\mathbb Z_p,\mathbb Z_2)$ is trivial for $p$ odd, proved that  if a lens space $L(p,q)$ contains a non-orientable surface then $p$ is even. Let $p=2k$ and 

\[\frac{2k}{q} = a_0 + \frac{1}{a_1 + \frac{1}{a_2 + \frac{1}{a_3 + \dots}}}\]

be a continued fraction representation of \(\frac{2k}{q}\). Moreover, we  set  $b_0 =a_0 $ and for $i$ a positive integer 

\begin{equation}\label{eqn:N_2k_q}
\left\{ \begin{array}{ll}
b_i = a_i &\text{ if } b_i\neq a_{i-1} \text{ or } \sum_{j = 0}^{i-1} b_j \text{ is odd.}\\
b_i = 0 &\text{ if } b_i = a_{i-1} \text{ and } \sum_{j = 0}^{i-1} b_j \text{ is even.}
\end{array}
\right.
\end{equation}
The following result (\cite[Section 6]{bredon1969non}) describes the non-orientable surfaces embedded in $L(2k,q)$. 
\begin{theorem}[\cite{bredon1969non}]
A  non-orientable closed surface $U_g$ of genus $g$ embeds in $L(2k,q)$ if and only if  $g=N(2k,q)+2h$, with $h\geq 0$, where
\begin{equation}N(2k,q) = \frac{1}{2} \sum_{i = 0}^{n} b_i.\label{formula:genus}\end{equation}
\end{theorem}

A minimal genus embedding has  been explicitly described in \cite{end1992non} in the following way. 

Present $L(2k,q)$ via the Heegaard splitting $\mathcal V_1\cup_{\psi_{k,q}}\mathcal V_2$, where $\mathcal V_1,\mathcal V_2$ are solid tori and $\psi_{k,q}$ is a gluing homemorphism between their boundaries that sends a meridian curve of $\mathcal V_1$ into the torus knot $t(2k, q)\subset \partial \mathcal V_2$. A family of embeddings, called Bredon-Woods embeddings, of $U_g$ into $L(2k,q)$ are constructed embedding  the compact surface  $F_g$, obtained removing the interior of a closed disk from  $U_g$,  into the solid torus $\mathcal V_1=D^2 \times S^1$, so that  $\partial F_g=t(2k, q) \subset S^1 \times S^1$, and gluing $F_g$ with a meridian disk of $\mathcal V_2$, via $\psi_{k,q}$. Let us describe how to construct embeddings of $F_g$ into $V_1\cong D^2\times S^1$.

Let $\xi = \textup{exp}(\frac{2\pi i}{p})\in S^1$. An involution $\tau$ of $\mathbb{Z}_{2k}$ is called \textit{non separating} if it has no fixed points, and if two the line segments $\overline{\xi^{a}, \xi^{\tau(a)}}$ and $\overline{\xi^{b}, \xi^{\tau(b)}}$ in $D^2$ are disjoint, for all $a,b\in\mathbb Z_{2k}$, unless $a=b$ or $a=\tau(b)$. 

Consider the pair  \((D^2 \times I, \partial(D^2 \times I)) \cong (D^3, S^2)\) and recall that for any (finite) system of pairwise non-intersecting simple closed curves on \(S^2\) we can  find mutually disjoint disks in \(D^3\) spanned by these curves. On \(\partial(D^2 \times I)\) take  the system of simple closed curves obtained by composing the line segments $\overline{\xi^{a}, \xi^{\tau(a)}} \times \{0\}$, $\overline{\xi^{a}, \xi^{\tau(a)}} \times \{1\}$ and the helices \(\{(\xi^{a+qt}, t) | t \in I\}, a = 1, \dots, 2k\). Let \(c = c(\tau)\) be the number of such simple closed curves and consider  $c$ mutually disjoint disks spanned by these curves in \(D^2 \times I\). Identifying \(D^2 \times \{0\}\) with \(D^2 \times \{1\}\) we obtain, after smoothing if necessary, a non-orientable, properly embedded compact surface \(F_g\) in \(D^2 \times S^1\), spanned by the torus knot, i.e. \(\partial F_g = t(2k,q)\). By \cite[Proposition 2.5]{end1992non} \(F_g\) has to be non-orientable, hence \(\chi(F_g) = 1-g\). Since \(F_g\) is obtained from \(c\) mutually disjoint 2-disks by identifying \(2k\) intervals on their boundaries pairwise, we get \(\chi(F_g) = c - k\). Therefore, the genus \(g\) of \(F_g\) is  given by $g = k + 1 - c$.\\

By \cite{bredon1969non} the minimal \(g\), obtained by varying \(\tau\), is also the minimal \(g\) for which \(U_g\) embeds into \(L(2k,q)\) so it is $N(2k,q)$. It is also the minimal \(g\), for which \(F_g = U_g \setminus \textup{int}(D)\) properly embeds into \(D^2 \times S^1\), with \(\partial(F) = t(2k,q)\).  Clearly, in order to get \(N(2k,q)\) one has to look for the maximal \(c(\tau)\). As proved in \cite{end1992non} it is always possible to obtain the minimal genus embedding as a Bredon-Wood embedding as described above and as proved in \cite{rubinstein1978one} the Bredon-Wood embedding induces a one-sided splitting of the lens space.           In the following two subsections we focus on two particular families of lens spaces depending on one parameter and we describe explicitly the embedded non-orientable surface as well as the handlebody structure of its complement.

\subsection{The case of \texorpdfstring{$L(2k,1)$}{TEXT}}\label{2k1}
Let $k>0$ be an integer, we have $2k/1=2k$, so by Formula \eqref{formula:genus},  the minimal genus surface has genus $k$. In order to understand the Bredon-Wood embedding of $U_k$ in  $L(2k,1)$ we need to find a $\tau$ function such $k=k+1-c(\tau)$, that is $c(\tau)=1$.  For $h=1,\ldots,2k$, label $h$ the point $\xi^h = \textup{exp}(\frac{2\pi h i}{2k})$ on $S^1$ and define $\tau$ as follows 
(see Figure \ref{fig:2k-1}) 
\[\tau(2i) = 2i-1, \quad \tau(2i-1) = 2i, \qquad i = 1, \dots, k.\]

\begin{figure}
   \centering
   \includegraphics[width = .4\textwidth]{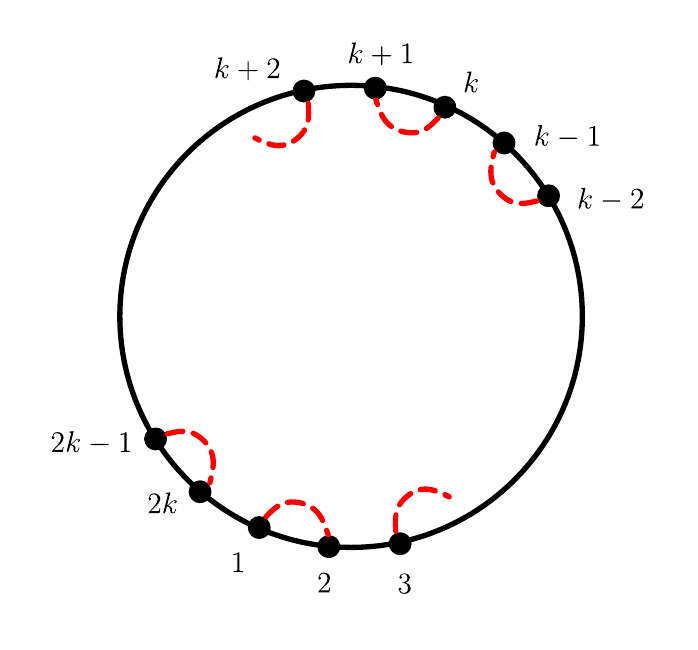}
   \caption{A minimal genus \(\tau\) function for the case of \(L(2k,1)\), for $k$ odd. When $k$ is even the label $k-2$ is replaced with $k-1$ and so on.}
   \label{fig:2k-1}
\end{figure}

The line segments  in $D\times\{0\}$ and $D\times\{1\}$ are the dashed arcs depicted in Figure \ref{fig:closed_tau_2k} having as endpoints the points  labeled with \(2i-1\) and   \(2i\), for  \(i = 1, \dots, k\).  Since $q=1$, each dotted-dashed helix connects two consecutive line segments, therefore all the dashed arcs and the dotted-dashed helices give rise to a single closed curve. 

\begin{figure}
    \centering
    \includegraphics[width = .4\textwidth]{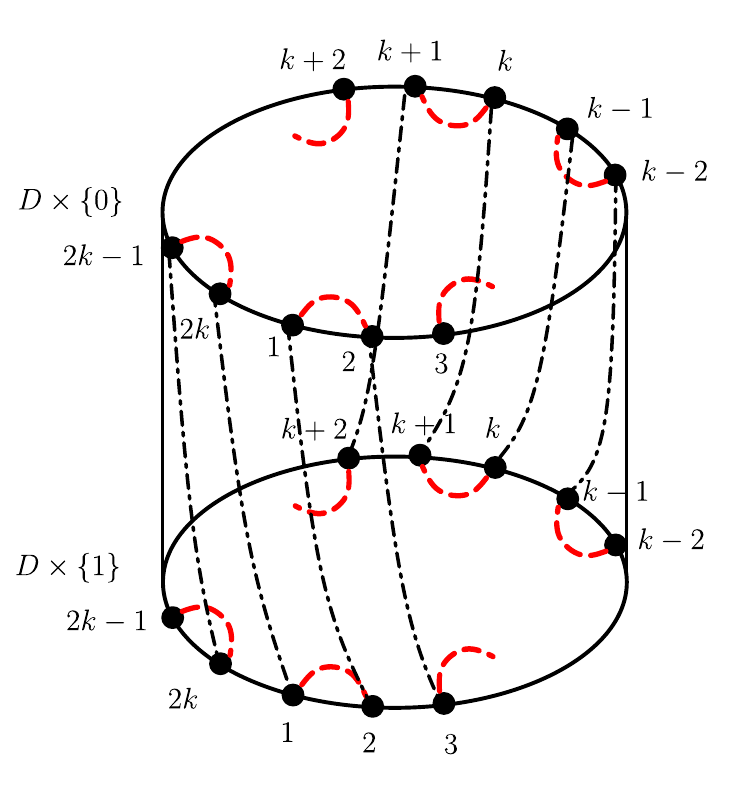}
    \caption{The single curve of a minimal genus \(\tau\) function for the case of \(L(2k,1)\), for $k$ odd. When $k$ is even the label $k-2$ is replaced with $k-1$ and so on.}
    \label{fig:closed_tau_2k}
\end{figure}

Identifying the boundary disks of the solid cylinder $D^2\times I$ shown, for the case $k=4$, in Figure \ref{fig:closed_tau_2k},  produces the solid torus $\mathcal V_1$. The blue surface in the solid cylinder will identify to form the non-orientable surface $F_k$ neatly embedded in $\mathcal V_1$. The surface consists of $k$ vertical rectangular regions and an horizontal region identified along the line segments in their boundaries.  The boundary of $F_k$ is  the  curve $\gamma$ consisting of the fatter blue lines on $\partial D^2 \times I$ and  is a $t(2k,1)$ torus knot on $\partial \mathcal V_1\cong S^1\times S^1$.  Notice that also the dotted lines  will form a parallel copy of $\gamma$, which also is an $(2k,1)$  torus knot in $\partial V_1$. Let's denote this parallel  curve with $\gamma'$.

\begin{figure}
    \centering    
    \includegraphics[width = .6\textwidth]{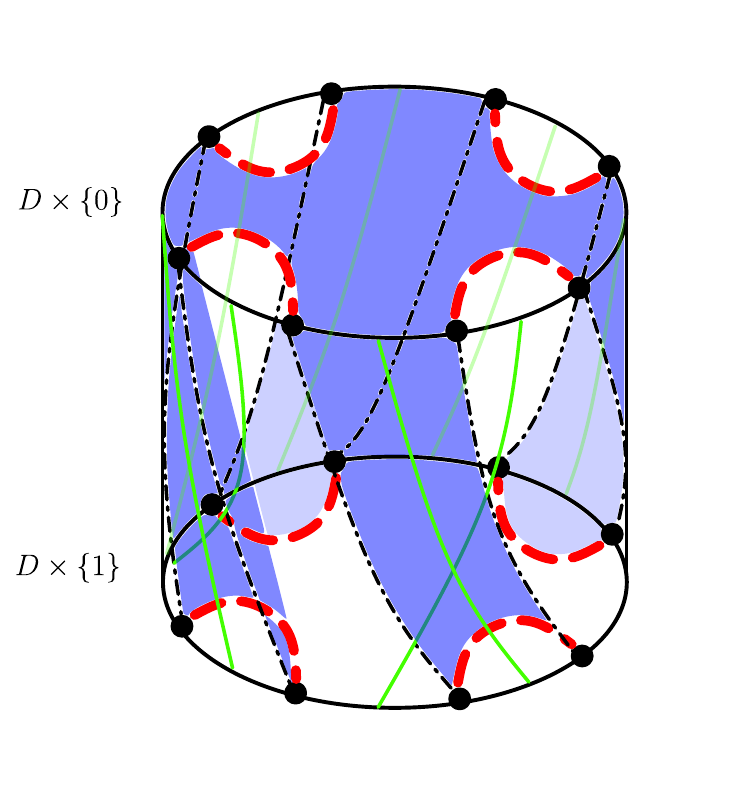}
    \caption{The surface $F_k$ (in blue), together with the core of the handles of $\mathcal V_1\setminus \mathcal N(F_k)$ (in green), in the case $k=4$. }    
    \label{fig:handlebody_2k_1}
\end{figure}

It follows from Theorem \ref{teo_rubinstein} that \(L(2k,1)\) is splittable. Indeed, we exhibit that the Bredon-Wood embedding induces a splitting, showing that the resulting complement handlebody has genus $k-1$. In order to show which are its handles, let us analyze what happens when we cut $\mathcal V_1$ along $F_k$. Referring to Figure \ref{fig:handlebody_2k_1}, notice that each of the blue rectangular regions of  the blue  surface, will cut the solid cylinder into two open sets. One of these will be a "half solid cylinder" lying in between the rectangular region and  the lateral surface.

Thus the complement of the surface in the solid cylinder will have $k$ copies of these half cylinders and a 3-ball, $B$. When we glue the two boundary disks of the solid cylinder to form $\mathcal V_1$, both of the ends of each of these half cylinders, will be identified to the ball.  Denote with $a_1,a_2,...,a_k$  the circles in $\mathcal V_1$  formed by joining the ends of the (green) lines contained in the half solid cylinder with an arc that belongs to $B$. The manifolds $\mathcal V_1\setminus \mathcal N(F_k)$ is a handlebody $\mathcal H$ of genus $k$, having $a_1,a_2,...,a_k$  as cores of the handles. The boundary of $F_k$ will be glued to the boundary of a meridian disk, say $D$, of $\mathcal V_2$ to form the surface $U_k$. So, the $L(2k, 1)\setminus \mathcal N(U_k)$ is obtained by gluing the  $2$-handle  $\mathcal V_2\setminus \mathcal N(D)$ to $\mathcal H$ along a curve $\gamma$ that is  parallel to the dashed-dotted curve (that is the torus knot $(2k,1)$). Since $\gamma$ is homologous in $\mathcal H$ to the sum of all $a_i$'s the resulting manifold is an handlebody of genus $k-1$.

\subsection{The case of \texorpdfstring{$L(4a+4, 2a+1)$}{TEXT}}
Let $a>0$ be an integer, applying formula \eqref{formula:genus}, we obtain:
\[\dfrac{4a+4}{2a+1} = \dfrac{2(2a+1)+2}{2a+1} = 2+\dfrac{1}{\dfrac{2a+1}{2}} = 2+\dfrac{1}{a+\dfrac{1}{2}} = [2, a, 2]\]
so that $b_0 = a_0 = 2, b_1 = 0$ and $b_2 = a_2 = 2$. This means that $N(L(4a+4, 2a+1)) = 4$, and the genus of the minimal genus  non-orientable surface embedded in $L(4a+4, 2a+1)$ is two and so it is a Klein bottle. Since in this case $k = 2a+2$ and $g=2$,  the number of cycles $c(\tau)$ is 
$c(\tau) = k-g+1 = 2a+2-2+1 = 2a+1.$
For $h=1,\ldots,4a+4$, label $h$ the point $\xi^h = \textup{exp}(\frac{2\pi h i}{4a+4})$ on $S^1$ and  define $\tau: \mathbb{Z}_{4a+4} \rightarrow \mathbb{Z}_{4a+4}$  as (see Figure \ref{fig:4a+4})
\[\tau(1) = 2a+2,\quad \tau(2i) = 2i+1, \quad \tau(2j+1) = 2j+2, \quad i = 1, \dots, a; j = a+1, \dots, 2a+1.\]

The line segments  in $D\times\{0\}$ and $D\times\{1\}$ are the dashed arcs depicted in Figure \ref{fig:closed_tau_4a+4}, whose endpoints are connected along the lateral surface of the cylinder through the dotted helices. 
We will have one "long" cycle containing  8 vertices, and $2a$ "short" cycles with  4 vertices, namely: 
\begin{itemize}
    \item Long cycle: $1 \times\{0\} \rightarrow (2a+2) \times\{0\} \rightarrow (4a+3) \times\{1\} \rightarrow (4a+4) \times\{1\} \rightarrow (2a+3) \times\{0\} \rightarrow  (2a+4) \times\{0\} \rightarrow 1 \times\{1\} \rightarrow (2a+2) \times\{1\} \rightarrow 1\times\{0\}$
    \item Short cycles with $i \in \{2, 2a\}$, $i$ even: $i \times\{0\} \rightarrow (i+1) \times\{0\} \rightarrow (2a+i+2) \times\{1\} \rightarrow (2a+i+1) \times\{1\} \rightarrow i \times\{0\}$
    \item Short cycles with $i \in \{2a+3, 4a+3\}$, $i$ odd: $i \times\{0\} \rightarrow (i+1) \times\{0\} \rightarrow (2a+i+2) \times\{1\} \rightarrow (2a+i+1) \times\{1\} \rightarrow i \times\{0\}$
\end{itemize}
a total of $2a+1$ cycles as required. 
As in the previous case, identifying the boundary disks of the solid cylinder $D^2\times I$ produces the solid torus $\mathcal V_1$. The surface in this solid cylinder will identify to form the non-orientable surface $F_2$ for any choice of \(a\), neatly embedded in $\mathcal V_1$. By Theorem \ref{teo_rubinstein}, \(L(4a+4, 2a+1)\) are splittable lens spaces. We show, as in Section \ref{2k1}, the structure of complement of  \(\mathcal N(U_2)\) in the lens space, namely \(L(4a+4, 2a+1)\), which is a handlebody of genus 1. In order to understand its handles, we can refer to Figure \ref{fig:surface_F2_4a+4}, where a part of the surface is shown. Similarly to the previous case, when we cut \(\mathcal V_1\) along this blue surface, we divide it into two pieces: one containing the center of the cylinder, and the other one running along its surface and lying in between the blue surface and the lateral surface. This component will be the first handle of our handlebody. The second one will be obtained similarly, starting from the longer $\tau$ arc \([1, 2a+2]\) at level \(0\) and following accordingly. In this way, the manifold \(\mathcal{V}_1 \setminus\mathcal{N}(F_2)\) is a handlebody \(\mathcal{H}'\) of genus 2. As before, the boundary of \(F_2\) will be glued to the boundary of a meridian disk \(D\) of \(\mathcal{N}_2\) to form the Klein bottle \(U_2\). So, the \(L(4a+4, 2a+1) \setminus\mathcal{N}(U_2)\) is obtained similarly gluing the 2-handle \(\mathcal{V}_2 \setminus \mathcal{N}(D)\) along a curve parallel to the cores of the handles (which are parallel to the torus knot \(4a+4, 2a+1\)). As before, this curve is homologous in \(\mathcal{H}'\) to the sum of the cores of the two handles, and the resulting manifold is an handlebody of genus 1. 

\begin{figure}
   \centering
   \includegraphics[width = .4\textwidth]{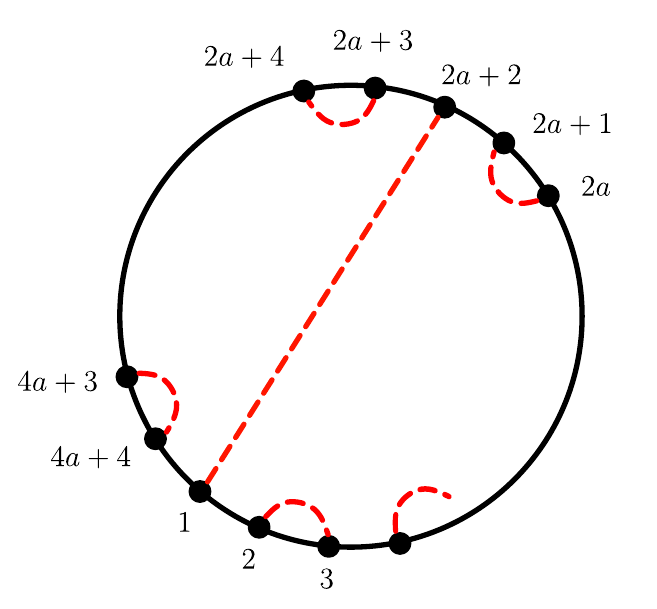} 
   \caption{A minimal genus \(\tau\) function for the case of \(L(4a+4, 2a+1)\).}
   \label{fig:4a+4}
\end{figure}

\begin{figure}
    \centering
    \includegraphics[width = .4\textwidth]{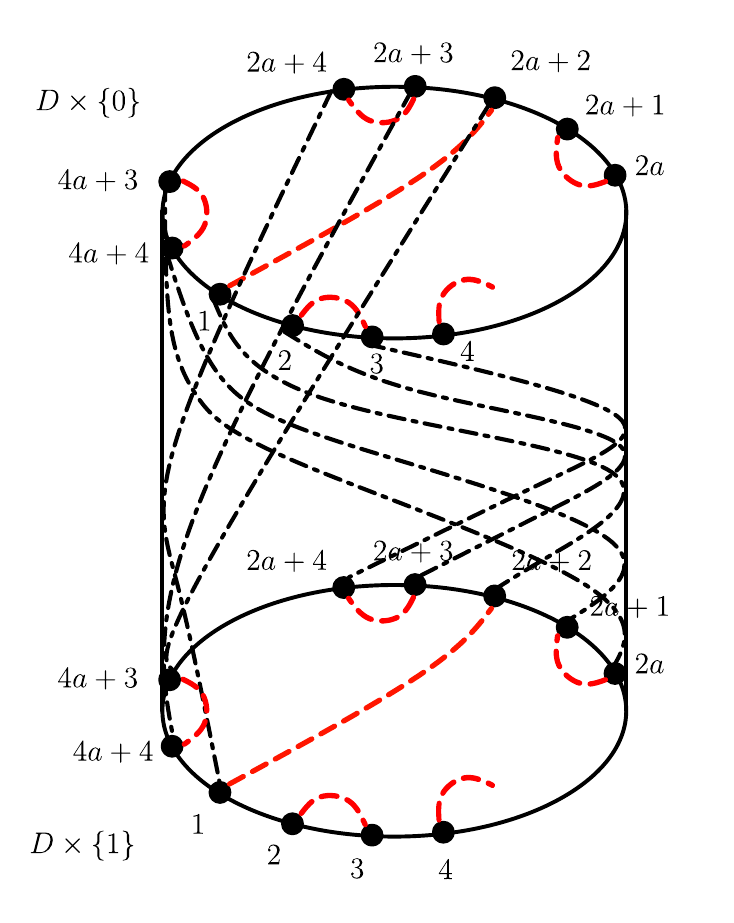}
    \caption{The collection of curves for a minimal genus \(\tau\) function for the case of \(L(4a+4, 2a+1)\). }
    \label{fig:closed_tau_4a+4}
\end{figure}

\begin{figure}
    \centering    
    \includegraphics[width = .5\textwidth]{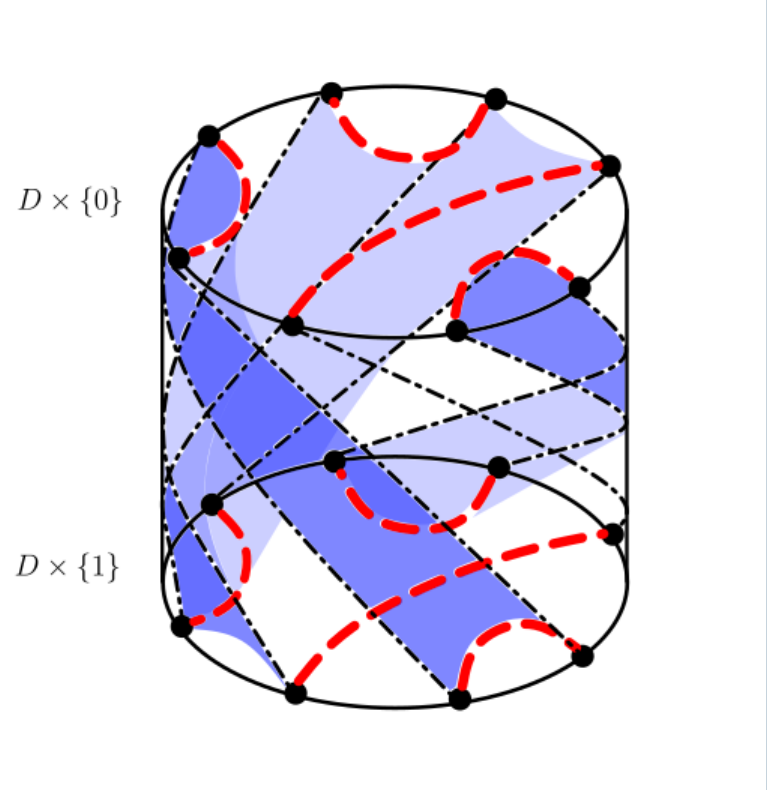}
    \caption{A part of the non-orientable surface \(F_2\).}    
    \label{fig:surface_F2_4a+4}
\end{figure}

\section{Manifolds of the type $\Sigma_g\times S^1$}\label{surfacebudlesection}

\begin{figure}
\centering
\includegraphics[width = .7\textwidth]{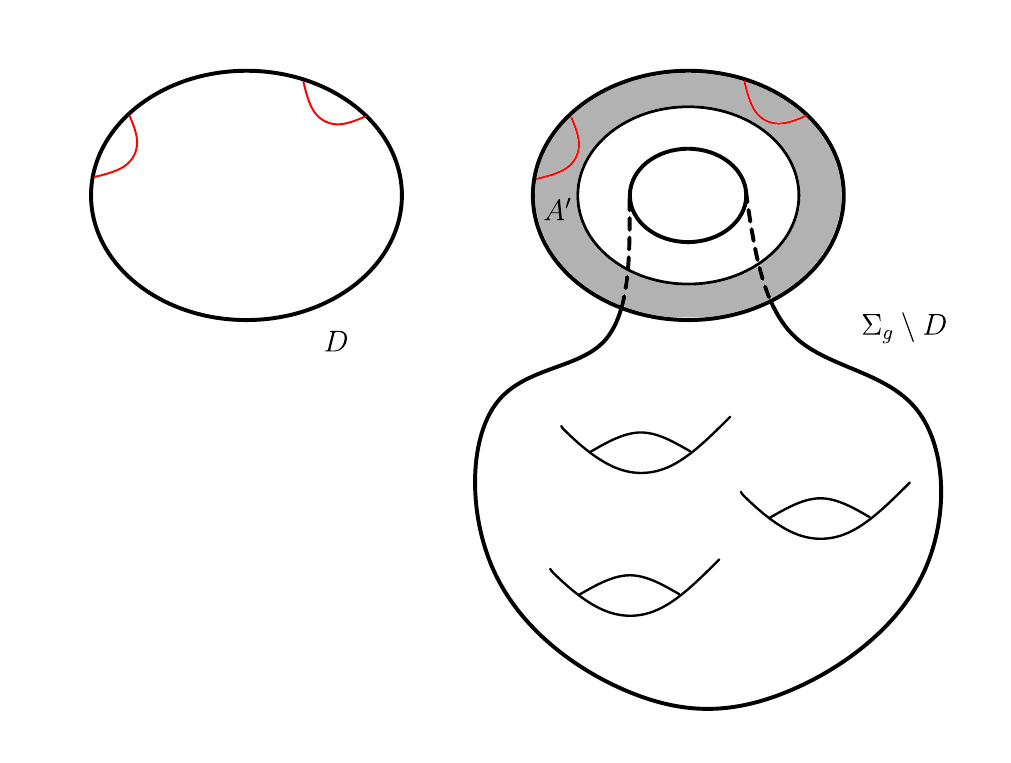}
\caption{The $\tau$ curves in a disk $D$ and its complement in $\Sigma_g$.}
\label{fig:surfacebundle}
\end{figure}

In this last section we deal with the 3-manifold $\mathcal{M}_g$ which is the trivial circle bundle over a genus $g$ orientable surface, $\Sigma_g$ and describe and embedded non-orientable surface inducing a one-sided splitting. 

Choose an arbitrary section, $\sigma: \Sigma_g\to \Sigma_g\times S ^1$ and denote with $p: \Sigma_g\times S ^1 \to \Sigma_g$ the projection. Let $c$ be a circle in $\Sigma_g$ bounding a disk $D$. In Figure \ref{fig:surfacebundle} there is a rough sketch of the two surfaces obtained by cutting $\Sigma_g$ along $\partial D$ in the case of $g=5$. Notice that the surface  on the right, that is $\Sigma_g\setminus \textup{int}(D)$, is homeomorphic to the thickening of a graph formed by $2g$ circles arranged in a row, where each circle meets the one on its right exactly at one point. The pull-back $p^{-1}(D)$ of the disk on the left in $\mathcal M_g$  is  a solid torus $\mathcal V=D\times S^1$ in this bundle. Let's denote the complement of $\mathcal V$ in $\mathcal{M}_g$ as $\mathcal V'$. Consider a $\tau$ function in $D$ as described in Section \ref{section_lens_spaces} and the corresponding arcs depicted in Figure  \ref{fig:surfacebundle} with  four points on $\partial D$ lying on  a (4,1)-knot in $\partial \mathcal V$. As described in Section \ref{2k1}, we can embed in $\mathcal V$ a compact non-orientable surface $F$, that is a Klein bottle with a disk removed so that: (1)  $F\cap \partial \mathcal V$ is the (4,1)-knot and $F\cap D$ are (2) the red arcs (as in Figure \ref{fig:handlebody_2k_1} but with two arcs instead of four). Consider the annulus $A'$ on the top part of the surface on the right side of Figure \ref{fig:surfacebundle}. The pull-back of $A'$ in $\mathcal V'$ is a thickened torus. We may construct a surface  $F'$ homemorphic to $F_{2g+2}$  inside $\mathcal V'$ in the following way. Notice that in the construction described in Section \ref{2k1} the M\"obius strips of $F$ are contained in a thickening of the boundary torus. So, if we consider the pull-back of the annulus $A'$ in $\mathcal V'$, we can construct two  M\"obius strips starting from the  red arcs and ending on $A'$. The union of the M\"obius strips together with the surface on the right of the picture is the non-orientable surface $F'$ in  $\mathcal V'$ homemorphic to $F\#\Sigma_g\cong F_{2g+2}$. By construction $F$ and $F'$ have a common boundary that is the (4,1)-knot on   $\partial \mathcal V=\partial \mathcal V'$. So if we glue  $F$ and $F'$  along it we obtain a closed connected non-orientable surface $U$ of genus $2g+4$. \\

Notice that $U$ can also be  described as follows. We take $\Sigma_g$, put four holes in it and add two cylinders on the newly created boundary circles connecting circles in couples. Each of these cylinders consists of two bands, one in $\mathcal V$ and the other in $\mathcal V'$. When joining back the cylinders in the pattern given by the $(4,1)$-knot of $\partial \mathcal V=\partial \mathcal V'$, each one of the bands will become part of a M\"obius strip. Hence,  the non-orientable genus of $U$ is $2g+4$. \\

% As shown in \cite{rubinstein1978one}, we have that this kind of manifolds are splittable: 

% \begin{theorem}\cite{rubinstein1978one}
% The manifolds $\mathcal M_g:= \Sigma_g\times S^1$ are splittable .
% \end{theorem}
% \begin{proof}
% Let $U\subset \mathcal M_g$ be the genus $2g+4$ non-orientable surface described above. It is enough to prove that $\mathcal M_g\setminus \mathcal N(U)$ is an handlebody of genus $2g+3$.

%  It is easy to see that $\mathcal V\setminus U$ is consists of three disjoint solid cylinders. This may be seen from the fact that $\mathcal V\setminus \Sigma_g$ is a solid cylinder on which the bands added to construct $U$ will carve out two solid cylinders. Similarly, $\mathcal V'\setminus U$ consists of a handlebody of genus $2g$ and two solid cylinders, since $\mathcal V'\setminus \Sigma_g$ is a handlebody of genus $2g$. The  handlebody in $\mathcal V'\setminus U$ is connected to one of the solid cylinders in $\mathcal V\setminus U$ through $2$ thickened strips in $\partial \mathcal V$. The two other solid cylinders carved out in $\mathcal V$ and those in  $\mathcal V'$  are joined along $\partial \mathcal V$ to form  two bigger solid cylinders, which are attached  as  two 1-handles to the previously described part. Thus, it is a clear that, on the whole, $\mathcal M_g\setminus \mathcal N(U)$ is a handlebody with genus $2g+3$.

%  \end{proof}

\medskip
\medskip
\medskip
\bigskip

\textbf{Acknowledgements}
\begin{itemize}
    \item The authors would like to thank J. H. Rubinstein and W. Wang for bringing to their attention the work \cite{rubinstein1978one, rubinstein1979, rubinstein1982} and for insightful discussions on the topic.
    \item Alessia Cattabriga has been supported by the "National Group for Algebraic and Geometric Structures, and their Applications" (GNSAGA-INdAM) and University of Bologna, funds for selected research topics.
    \item Paolo Cavicchioli has been supported by the Slovenian Research and Innovation Agency program P1-0292, the Slovenian Research and Innovation Agency grant J1-4031, Inštitut za matematiko, fiziko in mehaniko (IMFM) Ljubljana and University of Bologna. 
\end{itemize}

\bibliographystyle{alpha}
\bibliography{main}

\end{document}